\documentclass[letterpaper,USenglish]{lipics}
\usepackage{ourstyle}

\begin{document}

\title{On Minimum Area Homotopies of Normal Curves in the Plane}
\author[1]{Brittany Terese Fasy}
\author[2]{Sel\c{c}uk Karako\c{c}}
\author[3]{Carola Wenk}
\affil[1]{Gianforte School of Computing, Montana State University\\
  Bozeman, MT, USA\\
  \texttt{brittany@fasy.us}}
\affil[2]{Department of Mathematics, Tulane University\\
  New Orleans, LA, USA\\
  \texttt{skarakoc@tulane.edu}}
\affil[3]{Department of Computer Science, Tulane University\\
  New Orleans, LA, USA\\
  \texttt{cwenk@tulane.edu}}
  
\authorrunning{B.~T.~Fasy, S.~Karako\c{c}, and C.~Wenk} 
\Copyright{Brittany Terese Fasy, Sel\c{c}uk Karako\c{c}, and 
Carola Wenk}


\serieslogo{}
\EventShortName{}
\DOI{10.4230/LIPIcs.xxx.yyy.p}

\maketitle
\begin{abstract}
    In this paper, we
study the problem of computing a homotopy from a planar curve $C$ to a point that minimizes
the area swept.  The existence of such a {\em minimum homotopy} is a direct
result of the solution of Plateau's problem. Chambers and Wang studied the
special case that~$C$ is the concatenation of two simple curves, and they gave a
polynomial-time algorithm for computing a minimum homotopy in this setting. We
study the general case of a normal curve $C$ in the plane, and provide
structural properties of minimum homotopies that lead to an algorithm. In
particular, we prove that for any normal curve there exists a minimum homotopy that consists entirely
of contractions of self-overlapping sub-curves (i.e., consists of contracting a collection
of boundaries of immersed disks).

\end{abstract}



\section{Introduction}

The theory of minimal surfaces has been extensively studied by many mathematicians
and the existence of such surfaces with a given boundary, known as Plateau's
problem, has been proven by Rado and
Douglas~\cite{DG31,lagrange,HL80,plateau,RD30}. In this
work, we address the related problem of computing a minimum homotopy that
minimizes the homotopy area of a normal curve in the plane. Chambers and
Wang \cite{CW13} have defined the notion of minimum homotopy area to measure the
similarity between two simple curves that share the same start and end points.
Many continuous deformations, i.e., homotopies, between the two curves exist,
but a minimum-area homotopy is a deformation that minimizes the total area
swept. Chambers and Wang provided a dynamic programming algorithm to compute
such a minimum homotopy in polynomial time. 

Here, we study the more general task of computing the minimum homotopy area of
an arbitrary closed curve being contracted to a point; see \figref{Selfie} for
an example of such a minimum homotopy. This generalizes the Chambers and Wang
setting. One application would be to measure the similarity of two non-simple
curves (where we create a closed loop by concatenating the two curves).

Any normal homotopy can be described in terms of the combinatorial changes it incurs on the curve,
and can thus be characterized by a sequence of \emph{homotopy moves} \cite{JE13} which are projections of the
well-known Reidemeister moves for knots \cite{AB27}. 
In this paper, we provide structural insights for minimum-area homotopies. One of the
key ingredients is the use of self-overlapping curves \cite{BL67, EM09, MX74,SVW92,
TT61}. 
These curves are the boundaries of immersed disks and they have a natural interior.
An algorithm with a polynomial runtime has been given in~\cite{SVW92} to
detect whether a given normal curve is self-overlapping or not and to find the
interior of the curve in case it is self-overlapping. 
We show that the minimum homotopy area for a self-overlapping curve is equal to its winding area, the integral of the
winding numbers over the plane.

\begin{figure}[h] 
            \centering
	    \includegraphics{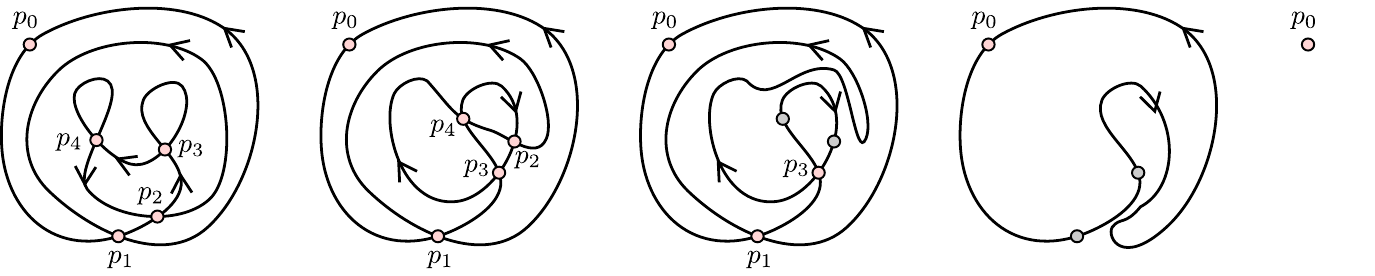} 
            \caption{A minimum homotopy is given as a sequence of homotopy 
                moves. The initial curve is self-overlapping. 
            }
    \label{fig:Selfie} 
\end{figure}

For a general normal curve, we show that a
minimum homotopy can be obtained by contracting a sequence of self-overlapping
subcurves that are based at
intersection points of the curve.  
This structural theorem reduces the space of homotopies to a finite candidate~set. 

In a preprint \cite{NIE14}, Nie provides an abstract algebraic construction for
computing the minimum homotopy. He reduces the problem to computing the weighted
cancellation distance on elements of the fundamental group induced by the planar
embedding, and this distance can be computed in polynomial-time using dynamic
programming. However, 
our approach is quite different and geometric in nature.

Our results not only solve the problem but also relate minimum homotopy to an interesting class of curves.
%

\section{Preliminaries}


In this section, we introduce the concepts of normal curves and homotopy moves, which we use throughout the paper.

\subsection{Normal Curves}

A \emph{closed curve} is a continuous map $C:[0,1] \rightarrow \R^2$ with $C(0)=C(1)$.
Let $[C]$ denote the image of this map. We call a closed curve \emph{(piecewise)
regular} if it is (piecewise) differentiable and its left and right derivatives
never vanish. Note that any regular curve $C$ is an immersion of the unit circle
$\S$ into $\R^2$. For a piecewise regular curve $C$, we call a point $p \in [C]$ 
an {\em intersection point} if $C^{-1}(p)$ consists of more
than a single point.  Without loss of generality, we assume that $C(t) \neq
C(0)$ for any $t \in (0,1)$.

An intersection point $p\in [C]$ is called a \emph{simple crossing point} if
there exist $t_1, t_2 \in [0,1]$, with $t_1\neq t_2$, such that
$C^{-1}(p)=\{t_1,t_2\}$ and if the tangent vectors at $t_1$ and $t_2$ exist and
are linearly independent.  In other words, a crossing point is simple if the
intersection is transverse.
A piecewise regular curve $C$ is called \emph{normal}
if it contains only a finite number of intersection points and these are all
simple crossing points. For a normal curve $C$, we define the complexity of $C$
as the number of simple crossing points. We set $P_C=\{p_0,p_1,\ldots,p_n\}$
where $p_0=C(0)$ and $p_i$ is an intersection point of $C$ for $i>0$.

Each normal curve $C$ naturally corresponds to a planar embedded directed graph
$G=(V,E)$. The vertex set $V=\{0, 1, \ldots, n\}$ represents the set of simple
crossing points $P_C=\{p_0, p_1, \ldots, p_n\}$, including the base point~$p_0$.
A directed edge $(i,j)\in E$ represents a direct connection along the curve from
$p_i$ to~$p_j$. 
We  call two  normal curves~$C_1$ and $C_2$
\emph{combinatorially equivalent} if their induced planar embedded graphs
\mbox{are isomorphic.}

Note that each face of this planar embedded graph corresponds to a maximal
connected component of~$\R^2 \setminus [C]$ whose boundary consists of a union
of edges of the graph.  Let $C$ be a normal curve and let $f_0, f_1, \ldots,
f_k$ be the faces of the induced graph defined by the image of $C$. For each $x
\in \R^2 \setminus [C]$, the \emph{winding number} of~$C$ at $x$, which we
denote as $wn(x,C)$, is defined as the signed number of times that the curve
`wraps around' $x$ \cite{CW13,JE13}. Notice that the winding number is constant
on each face $f$.  Thus, the winding area of a face $wn(f,C)$ is well-defined.
For
all $x \in [C]$, we define the winding area to be zero.

For a point $p_0\in\R^2$, let $\mathfrak{C}_{p_0}$ denote the set of all normal
curves with start point $p_0$, including the constant curve at $p_0$.
In the following, we only consider normal curves. Such an
assumption is justified, as Whitney proved that any regular curve can be
approximated with a normal curve that is obtained from an arbitrarily small
deformation~\cite{WH37}. 

The \emph{Whitney index} $\whit{C}$ of a regular normal curve $C$ is defined to
be the winding number of the derivative $C'$ about the origin. Note that, by
definition of a regular curve, the derivative $C'$ also defines a closed curve
and $(0,0) \not \in [C']$. The well-known Whitney-Graustein theorem \cite{WH37}
states that two regular curves are regularly homotopic if and only if they have
the same Whitney index.  For a piecewise regular closed curve $C$, we set
$\whit{C}=\whit{{\widetilde{C}}}$, where $\widetilde{C}$ is a regular curve
approximating~$C$, obtained by  smoothing the corners, i.e., the non-differentiable
points, of $C$ in an arbitrarily small deformation.

\subsection{Homotopies and Homotopy Moves}
Let $p_0\in\R^2$. A \emph{homotopy} between two curves $C_1, C_2\in
\mathfrak{C}_{p_0}$ is a continuous map  $H: [0,1]^2 \rightarrow \R^2$ such that
$H(0,t) =C_1(t)$, $H(1,t)=C_2(t)$, and $H(s,0)=p_0=H(s,1)$ for all $(s,t) \in
[0,1]^2$. A homotopy $H$ between~$C_1$ and $C_2$ is denoted as  $C_1
\homotopyarrow{H}  C_2$. Since $\R^2$ is simply connected, any two curves in
$\mathfrak{C}_{p_0}$ are homotopic. In particular, any curve
$C\in\mathfrak{C}_{p_0}$ is homotopic to the constant curve $p_0$. 

We concatenate two homotopies $C_1\homotopyarrow{H_1} C_2$ and
$C_2\homotopyarrow{H_2}  C_3$, denoted $H_1 + H_2 =: H$,  where the new homotopy 
is given as $H(s,t)=H_1(2s,t)$ for $t\in [0,\frac{1}{2}]$ and
$H(s,t)=H_2(2s-1)$  for $t\in [\frac{1}{2},1]$. Notice that $H$ is a homotopy
from $C_1$ to $C_3$.  Similarly, for a sequence of homotopies $\{C_i
\homotopyarrow{H_i}  C_{i+1}\}_{i=1}^k$, we denote their concatenation $C_1
\homotopyarrow{H}  C_{k+1}$  with $H= \sum _{i=1}^k H_i$.

Let $C \homotopyarrow{H} p_0$ be a homotopy. 
Consider an intermediate curve $\widetilde{C}$ of the homotopy.  For
each $p \in [\widetilde{C}]$, if $p$ is not an intersection point, then~$p$ 
neighbors two
faces of $\R^2 \setminus [\widetilde{C}]$.  We can use the orientation of the 
curve to define one
face to be the \emph{left face} and the other to be the \emph{right face}. 
We call $H$  {\em left sense-preserving} if for
any non-intersection point $p=H(s,t) \in [H]$, the point $H(s+\epsilon,t)$ lies
on or to the left of the oriented curve $H(s,\cdot)$ for each $s$ and $t$.
Similarly,~$H$ is {\em right sense-preserving} if $H(s+\epsilon,t)$ always lies 
on or
to the right of $H(s,\cdot)$.

As we deform normal curves using homotopies, we necessarily encounter non-normal
curves.  In order to stay within a nice family of curves, we define a curve to
be \emph{almost normal} if it has a finite number of intersection points, which 
are
either simple crossing points, triple points, or non-transverse (tangential)
crossing points.  We call a homotopy~$H$ from $C_1$ to $C_2$ a \emph{normal
homotopy} if each
intermediate curve is (piecewise) regular, and either normal or almost normal,
with only a finite set of them being almost normal.

Any normal homotopy $C \homotopyarrow{H}  p_0$ can be captured by
a sequence of homotopy moves which are similar to Reidemeister moves for knots.
There are three types of such moves: the \Ia- and \Ib-moves destroy/contract 
and create
self-loops; the \IIa- and \IIb-moves destroy and create regions defined by a
double-edge of the corresponding graph; and the \III-moves invert a triangle.
See \figref{SimpleTitusMoves}. Without loss of generality, we assume that at each event time point
there is only a single homotopy move. Any piecewise differentiable homotopy can be approximated by a normal 
homotopy \cite{FR71}.

\begin{figure}[htbp]
    \centering
    \subfloat[I-move]{\label{subfig:I-move} \includegraphics{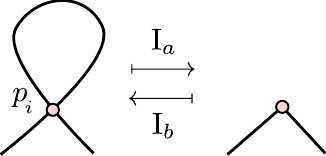}}
    \hspace*{7ex}
    \subfloat[\II-move] {\label{subfig:II-move} \includegraphics{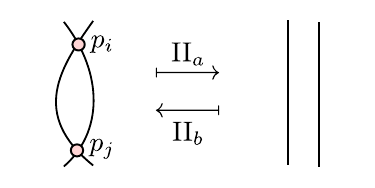}}
    \hspace*{7ex}
    \subfloat[\III-move] {\label{subfig:III-move} \includegraphics{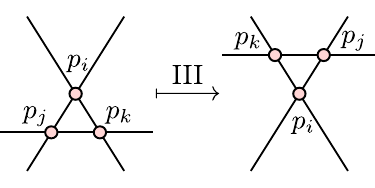}}
    
    \caption{Three types of homotopy moves are shown. An $a$-move destroys a
    monogon or bigon, a $b$-move creates a monogon or bigon, and a \III-move
inverts a triangle. 
}
   
    \label{fig:SimpleTitusMoves}
\end{figure}

\section{Minimum Homotopy Area}

In this section, we define minimum homotopy area and give basic properties of minimum homotopies.
We introduce self-overlapping curves and decompositions of~curves.

\subsection{Definition of Minimum Homotopy Area}
Let $C_1,C_2 \in \mathfrak{C}_{p_0}$ be two curves and $C_1 
\homotopyarrow{H} C_2$
be a homotopy. Let $E_H: \R^2 \rightarrow \Z$ be the function that assigns to
each $x \in \R^2$ the number of connected components of $H^{-1}(x)$. In other
words, $E_H$ counts how many times the intermediate curves $H(s)$ sweep over
$x$. The homotopy area $\textsc{Area}(H)$ of $H$ is defined as the integral 
of~$E_H$ over
the plane: 
$$\textsc{Area}(H)=\int_{\R^2} E_H(x) \text{ } dx.$$ 
Since addition distributes over the integral and since 
$E_{H_1+H_2}(x)=E_{H_1}(x)+E_{H_2}(x)$, the area is additive:
$\textsc{Area}(H_1+H_2)=\textsc{Area}(H_1)+\textsc{Area}(H_2)$.   

We define the minimum homotopy area between $C_1$ and $C_2$, denoted as
$\sigma(C_1,C_2)$, as the infimum homotopy area over all piecewise
differentiable homotopies between $C_1$ and $C_2$: 
$$\sigma(C_1,C_2)=\inf_H \textsc{Area}(H).$$ 
In this paper, we are interested in the special case where $C_2$ is the constant
curve.  Hence, we define the minimal homotopy area of a single curve~$C \in 
\mathfrak{C}_{p_0}$
to denote the minimal nullhomotopy of the curve~$C$, hence we write 
$\sigma(C):=\sigma(C,p_0)$.
    We note here that $\sigma(C)$ is well-defined,
    since $\sigma(C,p) = \sigma(C,p_0)$ for all $p \in [C]$.


A \emph{minimum homotopy} $H$ is a homotopy that realizes the above infimum. The
existence of minimum homotopies is a  result of Douglas' work on the solution of
Plateau's problem; see~\cite[Theorem $7$]{HL80}.

\begin{lemma}[Splitting a Minimal Homotopy]\label{lem:splitting}

    Let $C_1 \homotopyarrow{H_1} C_2$ and $C_2 \homotopyarrow{H_2}  C_3$ 
be two
    homotopies and $H=H_1+H_2$. If $C_1 \homotopyarrow{H}  C_3$ is a minimum
    homotopy, then we have: 

    \begin{enumerate}
        \item The sub-homotopies $H_1$ and $H_2$ are also
            minimum.\label{lemlist-subhomotopy}
        \item
            
$\sigma(C_1,C_3)=\sigma(C_1,C_2)+\sigma(C_2,C_3)$.\label{lemlist-sigmasplit}
        \item If $C_2 \homotopyarrow{H_2'} C_3$ is another minimum homotopy, so 
is
            $C_1 \homotopyarrow{H'} C_3$ where
            $H'=H_1+H'_2$.\label{lemlist-switch}
    \end{enumerate}
\end{lemma}

\begin{proof}
    Let $H=H_1+H_2$ be a minimal homotopy such that $C_1 \homotopyarrow{H_1}
    C_2$ and $C_2 \homotopyarrow{H_2} C_3$.
For the sake of contradiction, assume that~$H_1$ is not minimum. Then, there 
exists a homotopy $C_1
\homotopyarrow{H_1'} C_2$ such that $\textsc{Area}({H_1}')<\textsc{Area}(H_1)$. 
Define ${H}'={H}'_1+H_2$, and observe that $H'$ is a homotopy $C_1 
\homotopyarrow{H'} C_3$ such that
$$\textsc{Area}(H')=\textsc{Area}({H}'_1)+\textsc{Area}
(H_2)<\textsc{Area}(H_1)+\textsc{Area}(H_2)=\textsc{Area}(H).$$
However, the homotopy $H$ was minimum, so we have a contradiction.
Similarly, we can show that $H_2$ must be minimum, which proves Part
\ref{lemlist-subhomotopy} of this Lemma.

    Since $H$, $H_1$, and $H_2$ are minimal (from Part
    \ref{lemlist-subhomotopy}), 
    we know that $\sigma(C_1, C_3) =
\textsc{Area}(H)$,
$\sigma(C_1,C_2) = \textsc{Area}(H_1)$, and $\sigma(C_2,C_3) =
\textsc{Area}(H_2)$.  Putting this together, we conclude~$\sigma(C_1,C_3) = \textsc{Area}(H)= \textsc{Area}(H_1)+ \textsc{Area}(H_2)=
\sigma(C_1,C_2) + \sigma(C_2,C_3)$. This proves Part \ref{lemlist-sigmasplit}.
 
Let $C_2 \homotopyarrow{H_2'} C_3$ be a minimal homotopy, and let $H' = H_1 +
H_2'$.  Then, we know that $\textsc{Area}(H') = \textsc{Area}(H_1) +
\textsc{Area}(H_2')$.  Since $H_2$ and $H_2'$ are both minimal, we have
$\textsc{Area}(H_2) = \textsc{Area}(H_2')$, and so $\textsc{Area}(H') =
\textsc{Area}(H_1) + \textsc{Area}(H_2) = \textsc{Area}(H)$, thus proving Part
\ref{lemlist-switch} since $H$ is minimum. 
\end{proof}



\subsection{Winding Area}

The winding number defines a function 
$wn(\cdot,C):\R^2 \rightarrow \Z,$
where $wn(x,C)$ is the winding number of~$C$ around the point~$x$. 
We define the \emph{winding area}~$W(C)$ of $C$ as the integral:
$$W(C)=\int_{\R^2} |wn(x,C)| \,  dx.$$
Let $f_0,f_1,\ldots,f_k$ be the faces of $C$, where $f_0$ is the outer face. 
Since $wn(\cdot,C)$ is constant at each face of the curve and $wn(f_0,C)=0$, 
we obtain the following formula:
$$W(C)= \sum_{i=1}^k  |wn(f_i,C)| \cdot \textsc{Area}(f_i).$$

For example, consider the curve in  \subfigref{MinHomEx1}.  Here, we have
$W(C)=2\textsc{Area}(f_2)+\textsc{Area}(f_1)$, which is equal to the minimum 
homotopy area of the
curve. In general, the winding area is a lower bound for the minimum homotopy
area. This has been proved by Chambers and Wang for a special
class of curves~\cite{CW13}, but the same proof applies to our more
general~setting, which gives us the following lemma.

\begin{figure}[htbp]
 \centering
  \subfloat[$W(C)=\sigma(C)$]{\label{subfig:MinHomEx1} \includegraphics{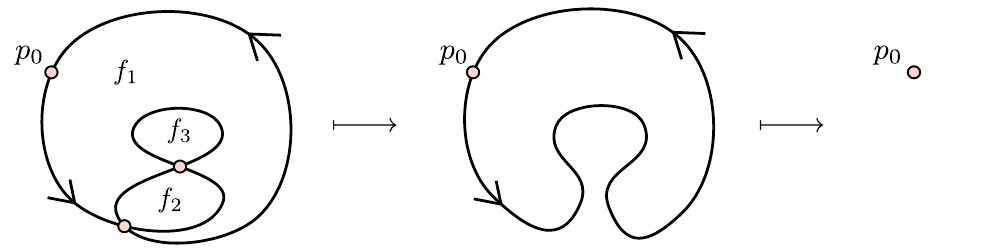}}
  \hspace*{.8cm}
  \subfloat[$W(C)<\sigma(C)$.]{\label{subfig:NotEqual}\includegraphics{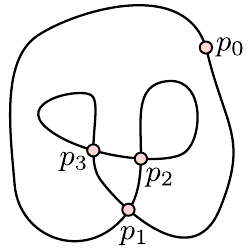}}

\caption{On the left, a minimum homotopy is given for the curve $C$, where 
$\sigma(C)=\textsc{Area}(H)=2\textsc{Area}(f_2)+\textsc{Area}(f_1)$. Notice that
$wn(f_1,C)=1, wn(f_2,C)=2$, and $wn(f_3,C)=0$. 
}
\end{figure}

\begin{lemma}[Winding Area Lower Bound]\label{lem:ChambersWang}
For any normal curve $C$, we have $W(C) \leq \sigma(C)$.
\end{lemma}

For some curves, the winding area is equal to the minimum homotopy area as in
\subfigref{MinHomEx1}. In Section 3.3, we define a class of curves 
for which the winding area equals the homotopy area.
In general, however, this equality does not hold, as is illustrated in
\subfigref{NotEqual}.

 A direct consequence of  \lemref{ChambersWang} is the following.

\begin{corollary}\label{cor:Equality}
If there exists a homotopy $C \homotopyarrow{H}  p_0$ such that 
$\textsc{Area}(H)=W(C)$, then $H$ is minimum and $\sigma(C)=W(C)$.
\end{corollary}

More generally, we have the following theorem.

\begin{theorem}[Sense-Preserving Homotopy Area]\label{thm:Sense}
    If $C \homotopyarrow{H} p_0$ is a sense-preserving homotopy, then $H$ is minimum and
$W(C)=\sigma(C)=\textsc{Area}(H)$. Similarly, if $\textsc{Area}(H)=W(C)$, then~$H$ is sense-preserving.
\end{theorem}

\begin{proof}
Without loss of generality, assume that $H$ is left sense-preserving.
Consider $H(s,\cdot)$ as $s$ ranges over $[0,1]$.  Let $x \in \R^2 - [C]$.
Then, each time $x \in [H(s,\cdot)]$, the winding number at $x$ decreases by
one.  Hence,
$E_H(x)=wn(x,C)$  and $W(C)=\textsc{Area}(H)$, or $H$  is a
minimum homotopy and $W(C)=\sigma(C)$ by \corref{Equality}.

On the other hand, if $H$ is not sense-preserving, then there is a region $R$ that is swept by edges moving left and edges moving right. Hence, if  $x \in R$, $E_H(x)> |wn(x,C)|$. In other words $\textsc{Area}(H) >W(C)$. 
\end{proof}

\subsection{Self-Overlapping Curves and $k$-Boundaries}

Chambers and Wang introduced the notion of consistent winding numbers to
describe a class of curves for which the homotopy area and the winding area are 
equal. In this
subsection, we identify a more general class of closed curves for which the 
same equality is satisfied.

A regular normal curve $C\in\mathfrak{C}_{p_0}$ is \emph{self-overlapping} if
there exists an immersion of the two-disk $F: D^2 \rightarrow \R^2$  such that
$[C]=F|_{\partial D^2}$. If $C$ is not a regular normal curve, then we call~$C$
self-overlapping if there exists an arbitrarily-close approximation $\tilde{C}$,
where $\tilde{C}$ is a regular normal curve that is self-overlapping. The image
$F(D^2)$ is called the \emph{interior} of $C$.  Self-overlapping curves have
been investigated in~\cite{BL67,MX74,SVW92,TT61}. A dynamic programming
algorithm for testing whether a given curve is self-overlapping has been given
in \cite{SVW92}; the runtime of this algorithm is cubic in the number of
vertices of the input polygon.
Examples of self-overlapping curves are given in \figref{Selfie},
\figref{Milnor} and \subfigref{MinHomEx1}. The curve in \subfigref{NotEqual} is
an example of a curve which is not self-overlapping.
In the following theorem, we prove that the homotopy area equals the winding
area for self-overlapping curves.

\begin{theorem}[Winding Area Equality for Self-Overlapping Curves]\label{thm:Self}
If $C \in \mathfrak{C}_{p_0} $ is a self-overlapping curve, then $\sigma(C)=W(C)$.
\end{theorem}

\begin{proof}
    A straight-line deformation retract $r_s \colon D^2 \to D^2$ 
    from the unit disk $r_0(\cdot) = D^2$ to a point
    $r_1(\cdot) = q_0\in \S^1=\partial D^2$ induces a minimum-area homotopy $H$ 
from $\S^1$
to $q_0$, for which $E_H(x)=1$ if $x \in D^2$ and  $E_{H}(x)=0$ otherwise.

Now, let $F: D^2 \rightarrow \R^2$ be an immersion of the disk such that
$F(q_0)=p_0$ and the curve $C=F|_{S^1}$ is normal.
Then, the composition $f_s = F \circ r_s$ is a deformation retract of the 
immersed disk
$f_0 = F(D^2)$  to $f_1 = p_0$.  
Restricting this map to $\S^1$, we obtain a homotopy 
$H(s,t)=f_s(e^{i 2t\pi})$
%
%
for which $H(0,\cdot) = C$ and $H(1,\cdot) = p_0$.
Moreover, we have
$E_{H}(x)=|F^{-1}(x)|=wn(x,C)$; a proof of this can be found in~\cite{CS66}. 
In other words, the homotopy $H$ satisfies $\textsc{Area}(H) = W(C)$.
Hence, by \corref{Equality}, we conclude that $H$ is a minimum homotopy with 
$\sigma(C)=\textsc{Area}(H)=W(C)$.
\end{proof}

If $C$ is regular and normal, then the homotopy defined in the proof of 
\thmref{Self} is
regular. Furthermore, the intermediate curves eventually become simple loops
with Whitney number  $\pm1$.  Hence, by the Whitney-Graustein  Theorem
\cite{WH37}, we know that $\whit{C}= \pm 1$ for a self-overlapping curve $C$. 
We 
call a
self-overlapping curve positive if $\whit{C}=1$, and otherwise we call it 
negative. 
Observe that, by definition, the Jacobian of an immersion of the disk 
is either always positive or always negative. Hence, a self-overlapping curve 
is 
positive (or negative) if it can be extended to an immersion whose Jacobian is 
always positive (resp., negative).
We summarize this with the following lemma:

\begin{lemma}[Equivalent Properties for Self-Overlapping Curves]\label{lem:equivalences}
 The following statements are equivalent for a regular self-overlapping curve $C$:
 \begin{itemize}
  \item $\whit{C}=1$
  \item  An immersion $\S \to C$ can be extended to 
      an immersion of $D^2$ whose Jacobian is always positive.
  \item The winding numbers $wn(C,p)$ are non-negative for all $p \in \R^2$.
  \item If $C \homotopyarrow{H} p$ is a minimum homotopy for 
      some $p \in [C]$, then $H$  is left sense-preserving.
 \end{itemize}
\end{lemma}


\begin{observation}\label{obs:whitneyInvariant}
The Whitney index of a curve is invariant under \II- and
\III-moves, but not under \I-moves.  
Moreover,
a regular homotopy uses only \II- and
\III-moves, hence the complexity of a self-overlapping curve (defined above to be the
number of simple crossing points) is always even.
\end{observation}

\begin{definition}[Decomposition]
A \emph{decomposition} of a normal curve $C$ is a set
$\Gamma=\{\gamma_i\}_{i=1}^l$ of closed subcurves of $C$ such that:
\begin{itemize}
\item Each $\gamma_i$ is self-overlapping.
\item For each $i \neq j \in \{1,2,\ldots,l\}$, $[\gamma_i] \cap [\gamma_j] \in 
P_C $.
\item $\cup_{i=1}^l [\gamma_i] =[C]$.
\end{itemize}
\end{definition}

Such a
decomposition always exists for the following reason.  Each curve contains a
self-overlapping loop. If we remove this self-overlapping loop
from the curve, we still have a closed curve which contains another
self-overlapping loop.
Continuing this process, we decompose the curve into self-overlapping loops.  An example
of a decomposition of a curve is given in \figref{MinHomEx3}. For each~$\gamma
\in \Gamma$, we define the \emph{root} $p\in P_C$ of $\gamma$ as follows when
$p_0 \notin \gamma$: 
$p := C\left( \inf \{t \in [0,1] : C(t) \in [\gamma]\} \right)$. 
If $p_0 \in \gamma$, then we define the root of $\gamma$ to be the root of
the complement~$C\setminus\gamma$.

For any decomposition, there exists an ordering
$\Gamma=\{\gamma_1,\gamma_2,\ldots,\gamma_k\}$ such that the root of~$\gamma_i$
does not appear in $\gamma_j$ for any $j \geq i$.
Thus, the decomposition $\Gamma$ defines a
homotopy~$H_{\Gamma}$ which can be obtained by contracting each subloop
$\gamma_i \in \Gamma$ to its roots, as in \thmref{Self}, starting from the last
subcurve $\gamma_k$ to the first subcurve $\gamma_1$. 
If $C \in \mathfrak{C}_{p_0}$ admits a decomposition $\{\gamma_1, \ldots,
\gamma_k \}$, where
each $\gamma_i$ is positive,  we call $C$ a $k$-boundary.
These curves
have been investigated by Titus \cite{TT61}, where he calls such curves \emph{interior
boundaries}. He also gives an algorithm to detect whether a given curve is a
\mbox{$k$-boundary.}

\begin{observation}\label{obs:kBoundary}
If $C$ is a $k$-boundary with  decomposition $\Gamma = \{ \gamma_1, \gamma_2, \ldots,
\gamma_k\}$, then 
$\textsc{Area}(H_{\Gamma}) = \sum_{i=1}^{k} W(\gamma_i) = \sum_{i=1}^{k} 
\sigma(\gamma_i)$.  By \corref{Equality}, we conclude that $H_{\Gamma}$ 
is a well-behaved minimum homotopy, as defined in the next section, and $\sigma(C) = W(C) = \textsc{Area}(H_{\Gamma})$.
In addition,
if $C$ is a two-boundary with decomposition $\{\gamma_1, \gamma_2\}$, then
$\sigma(C) = \sigma(\gamma_1) + \sigma(\gamma_2)$.
\end{observation}

We call the curve $C$ a (-k)-boundary if the inverse of the curve
$C^{-1}$ is a k-boundary. Such curves admit a decomposition where each
self-overlapping subloop is negative.
More generally, 
we observe that $\whit{C}= \sum_{i=1}^l \whit{\gamma_i}$ and
$wn(x,C)=\sum_{i=1}^l wn(x,\gamma_i)$ for each point~$x$ in the plane. Hence,
$W(C) \leq \sum_{i=1}^l W(\gamma_i)$ and $\sigma(C) \leq
\textsc{Area}(H_{\Gamma}) = \sum_{i=1}^l
\sigma(\gamma_i)$.

\begin{theorem}[Minimum Homotopy Decomposition]
    Let $C$ be a self-overlapping curve. If~$\Gamma $ is a decomposition with
    $|\Gamma|>1$, then the induced homotopy $H_{\Gamma}$ is not
    minimum. Likewise, if~$C$ is a k-boundary and~$\Gamma$ is a decomposition
    of~$C$ with $|\Gamma|>k$, then $H_{\Gamma}$ cannot be minimum.
\end{theorem}

\begin{proof}
We prove the base case for a proof by induction.
Let $C$ be positive and self-overlapping, and let 
$\Gamma=\{\gamma_1,\gamma_2,\ldots,\gamma_{\ell}\}$ be a decomposition with 
$\ell>1$. Then, there exists a 
negative self-overlapping subcurve $\gamma \in \Gamma$ and a positive 
self-overlapping subcurve $\widetilde{\gamma} \in \Gamma$, since 
$1=\whit{C}=\sum_{i=1}^{\ell} \whit{\gamma_i}$ and $\whit{\gamma_i}=\pm1$.
Observe that the induced homotopy should be right sense-preserving  on $\gamma$  
and left 
sense-preserving on $\widetilde{\gamma}$. In other words, the total homotopy is not sense-preserving.
Thus, by \thmref{Sense}
$\textsc{Area}(H_{\Gamma})>W(C)$.
This implies that $H_{\Gamma}$ is not minimum by \thmref{Self}.
The second half of the proof follows from a simple inductive argument.
\end{proof}

\section{Construction of a Minimum Homotopy}

In this section, we  prove our main theorem which states that each normal
curve $C$ admits a decomposition $\Gamma$ such that the induced homotopy
$H_{\Gamma}$ is minimum. 

\subsection{Well-behaved Minimum Homotopies}

Let $C\in \mathfrak{C}_{p_0}$ be a curve and let $P_C=\{p_0,p_1,\ldots,p_n\}$ be
its set of simple crossing points. Let $C \xrightarrow{H} p_0$ be a homotopy. Observe that when we
perform the homotopy, each simple crossing point moves continuously following 
the
intersection points of intermediate curves until those simple crossing points 
are
eliminated. Assume for now that $H$ does not create new simple crossing points, 
i.e., it
does not contain any $b$ moves.  Then, each simple crossing point $p_j \in P_c$
is eliminated via either a \Ia \text{ }or a \IIa-move. 
We call a
index $j$ an \emph{anchor index} if $p_j$ is eliminated via a \Ia-move; in this 
case, we call $p_j$ the corresponding \emph{anchor point}.
Similarly, when a new intersection point $p_j$ is created by a $b$ move, $j$
is called an anchor index, if it is later eliminated by a \Ia-move.

\begin{figure}[htbp]
\centering
{\includegraphics{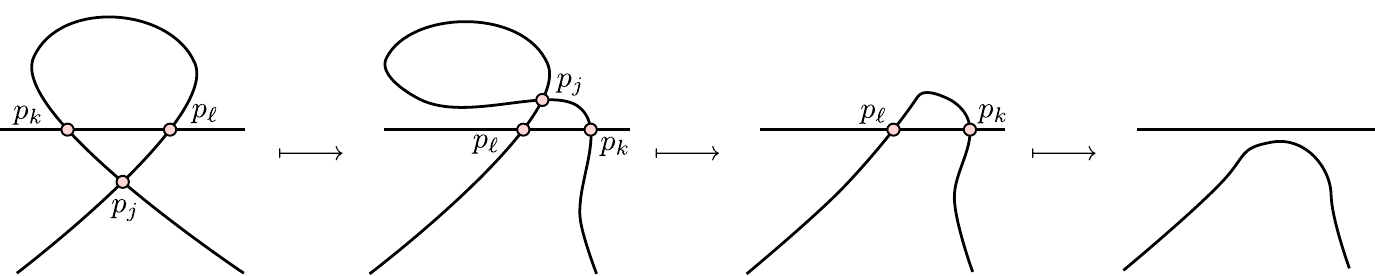}\label{fig:AnchorIndex}}
\caption{For this homotopy, $p_j$ is an anchor point, since it is removed with
    a \Ib-move.  On the other hand, $p_k$
    and $p_l$ are not anchor points as they are removed with a \IIa-move.}
\end{figure}

For a homotopy $C \homotopyarrow{H} p_0$, we define $A_H=\{j: p_j \text{ is 
an
anchor point}\}$. We order $A_H$ according to the time the vertices are
destroyed. Notice that $p_0$ is always an anchor point since the last move
for each homotopy is a \Ia-move which contracts an intermediate curve which is 
a 
simple loop.

At first glance, one may think that minimum homotopies should only decrease the
complexity of the graph of the curve, and that $b$-moves increase the
complexity. Naturally, one may conjecture that each curve has a minimum homotopy
without any $b$-moves. However, there are curves for which this is not true.
Consider for example the Milnor curve shown in \figref{Milnor}. For this curve,
any minimum homotopy has to contain a \IIb-move. (This curve is, in fact,
self-overlapping, and a minimum homotopy that is indicated by the shading 
sweeps an area equal to the winding area.)  In the following, we show that
these particular \IIb-moves do not create any complications.  Let $C_1
\homotopyarrow{H_1}  C_2 \homotopyarrow{H_2} p_0$ be a homotopy such that 
$H_1$
consists of a single \IIb-move, thus creating two new intersection points. Then,
we say that the \IIb-move is \emph{significant} if either of the intersection points  that is created by the move is an anchor point of
$H_2$. Intuitively, a significant \IIb-move makes a structural impact itself,
while an insignificant \IIb-move is only an intermediate move that allows a
portion of the curve to pass over another portion. We call a minimum homotopy
{\em well-behaved} if it does not contain any \Ib-moves or significant
\IIb-moves.  

\begin{figure}[htbp] 
    \centering
    \includegraphics{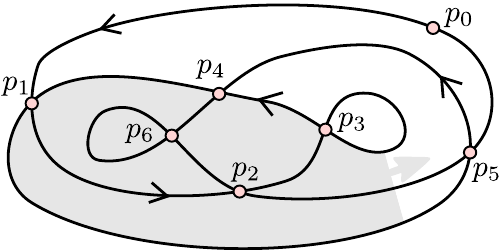} 
    \caption{This curve is
        self-overlapping and it does not admit a minimum homotopy without
        \IIb-moves. The shading indicates an immersion that also
        defines a minimum homotopy.} \label{fig:Milnor}
\end{figure}

\begin{lemma}\label{lem:SensePreserving}
Let $C \homotopyarrow{H}  p_0$ be a minimum homotopy which has a 
single anchor point $p_0$, then $H$ is sense-preserving. 
\end{lemma}
 The proof of \lemref{SensePreserving} is identical to the proof of Lemma 3.2 
in 
\cite{CW13}.

\begin{lemma}\label{lem:OneMove}
Let  $C_1 \homotopyarrow{H_1}  C_2 \homotopyarrow{H_2} p_0$ be a minimum 
homotopy, where $H_1$ consists of a single homotopy move that is either a \IIa-move, an
insignificant \IIb-move or a \III-move. Then, the curve $C_1$ is a positive 
self-overlapping curve if and only if $C_2$ is a positive self-overlapping~curve. 
\end{lemma}

\begin{proof}

Notice that both $H_1$ and $H_2$ are left sense-preserving.  Divide the disk
$D^2$ into two regions, $W$ and $E$ with a line segment $L$.	The different
homotopy moves induce regions~$W', E'$ and curve segment $L'$ in $C_1$ as shown
in \figref{NonEffective}, and since $\whit{C_1}=1$ the region~$W'$ always lies
in the interior of $C_1$. 
 
 \begin{figure}[htbp]
 \centering
\includegraphics{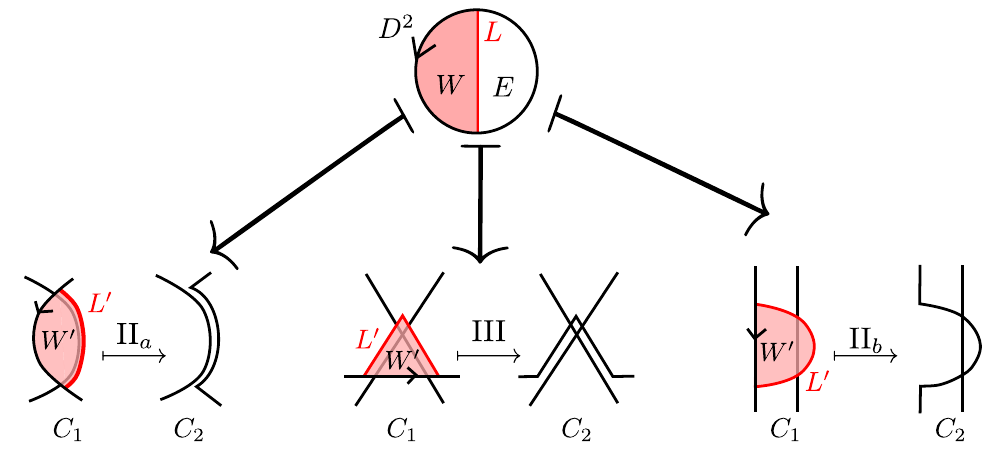}
 \caption{Here, we map the region $W$ of the disk to corresponding region 
$W'$ under the three types of homotopy moves induced by a deformation retraction
on $D^2$.
}\label{fig:NonEffective}
\end{figure}
	
Now, if $C_1$ is positive self-overlapping,  we can find an immersion $F: D^2
\rightarrow\R^2$ that maps~$W$ to $W'$ and $L$ to $L'$, and the restriction of
$F$ to $E$ gives an immersion whose boundary is~$C_2$.
Similarly, if $C_2$ is self-overlapping, then there is an immersion $G:
E\rightarrow \R^2$  that maps~$L$ to~$L'$. We can extend $G$ to $D^2$ by
mapping~$W$ to $W'$ so that $\partial W=L$ is mapped to $\partial W'=L'$. The 
extended
immersion sends the boundary of the disk to $C_1$.
\end{proof}

The theorem below follows from the previous two lemmas, and is illustrated in
\figref{Selfie}.

\begin{theorem}\label{thm:OneAnchor}
Let $C$ be a normal curve with a minimum homotopy $C \homotopyarrow{H}  p_0$. 
If $H$ is  well-behaved and has a single anchor point $p_0$, then $C$ is 
self-overlapping. Furthermore, each intermediate curve  is self-overlapping.
\end{theorem}
\begin{proof}
Notice that before the last homotopy move of $H$, the intermediate curve is a 
simple loop. Simple loops are self-overlapping. Hence by  \lemref{OneMove} each 
intermediate curve is self-overlapping. Therefore, $C$ is itself 
self-overlapping. 
\end{proof}

For a well-behaved homotopy with more than one anchor point, we have the 
following:

\begin{theorem}\label{thm:ReconstructionStable}
Let $C$ be a normal curve which admits a well-behaved  minimum homotopy $C 
\homotopyarrow{H}  p_0$. Then, there is a corresponding decomposition 
$\Gamma$ of $C$ such that $\textsc{Area}(H)=\textsc{Area}(H_{\Gamma})$. Hence, 
$H_{\Gamma}$ is also a minimum  homotopy for $C$.
\end{theorem}

\begin{proof}
Let $C\homotopyarrow{H}  p_0$ be a well-behaved homotopy. If $|A_H|=1$, then 
$C$
is self-overlapping by \thmref{OneAnchor}. In other words,
$\Gamma=\{C\}$ and the theorem follows. Hence, we assume that $|A_H|>1$. Let
$A_H=\{i_1,i_2, \ldots, i_k\}$. Consider the first anchor index $i_1$.  Let
$\gamma_1$ be the subcurve of $C$ based at the intersection $p_{i_1} \in [C]$.
Since $H$ restricted to $\gamma_1$ only has $i_1$ as an anchor index, it follows
from \thmref{OneAnchor} that $\gamma_1$ is self-overlapping. We define $H_1$ to
be the homotopy that contracts $\gamma_1$ linearly as in Theorem \ref{thm:Self}.
We denote the remaining curve $C_1$. Analogously, we consider $i_2$ and its
corresponding subcurve $\gamma_2$ of $C_1$, which is also a subcurve of $C$.
Then we define $H_{2}$ by contracting $\gamma_2$ in a similar fashion to
obtain a curve which we denote $C_2$. 
	
Continuing this way, we are left with a self-overlapping curve
$\gamma_k=C_{k-1}$ based at $p_0$ which we can contract to the point $p_0$ in a
similar way. Hence, we constructed a decomposition $\Gamma=\{\gamma_i,
i=1,2,\ldots,k\}$ of $C$. The homotopy $H_{\Gamma}$ sweeps each point of the
plane no more than $H$ does. In other words, 
$\textsc{Area}(H_{\Gamma})=\textsc{Area}(H)$ and
$H_{\Gamma}$ is also a minimum homotopy.
\end{proof}

An immediate corollary of  \thmref{ReconstructionStable} is the following:

\begin{corollary}\label{cor:Kboundary}
A curve $C$ is a k-boundary if and only if it admits a left sense-preserving 
homotopy with $k$ anchor points.
\end{corollary}
\begin{proof}
If $C$ is a k-boundary, then $C$ admits a decomposition with $k$ positive self-overlapping subcurves. Contracting each of them to the corresponding roots gives a left sense-preserving homotopy with $k$ anchor points.On the other hand, if the homotopy is left sense-preserving with $k$ anchor points, then it decomposes the curve into
$k$ positive self-overlapping~subcurve.

\end{proof}
\subsection{Main Theorem}
To prove our main theorem, we show that there exists a well-behaved minimum 
homotopy, i.e., a homotopy that
does not contain any \Ib-move or a significant \IIb-move. This is done by 
taking an arbitrary minimum homotopy and
eliminating any such $b$-moves one by one starting from the last one. Hence, our
main theorem follows from \thmref{ReconstructionStable}. 

Now, we prove a technical lemma. Let $C$ be a curve and let $[\gamma]=\{C(t): t \in [t^*,t^{**}]\}$ be a simple subloop with the root $p=C(t^*)=C(t^{**})$. Denote $[C \backslash \gamma ]= \{C(t): t \in [0,1] \backslash (t^*,t^{**})\}]$. In other words, $C \backslash \gamma$ is the curve obtained from $C$ by removing the simple loop $\gamma$.

\begin{lemma}[Decomposing Self-Overlapping Curves]\label{lem:SemiSelf}
Let $C$ be a positive self-overlapping curve. If there is a simple 
subloop $\gamma$ of $C$ which is negative, i.e., $\whit{\gamma}=-1$, then the 
curve obtained by contracting $\gamma$ via a \Ia-move is a two-boundary.
\end{lemma}


\begin{proof}
    Let $p\in[C]$ be the root of $\gamma$, and let $C_1$ be the curve obtained
    from
    $C$ by contracting~$\gamma$ via a \Ia-move. We observe that 
    $\whit{C_1}=\whit{C}-\whit{\gamma}=1-(-1)=2$.

     \begin{figure}[htbp]
          \centering
          \includegraphics{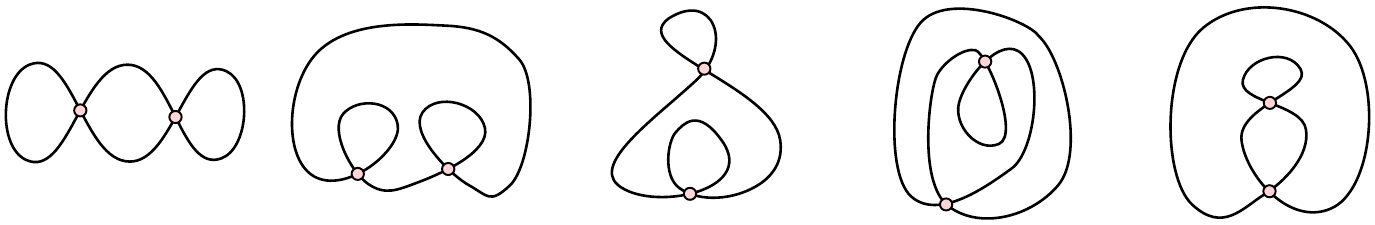}
           \caption{The complete set of normal curves with exactly two
               intersection 
               points, up to 
                    (planar) graph isomorphism. The first four curves are
                         non-self-overlapping; whereas, the 
                     the rightmost curve is
                 self-overlapping.}\label{fig:ComplexityTwo}
             \end{figure}
             First, consider the case where $C$ has exactly two intersection
             points.  In this
             case, there are five unique normal curves, up to planar graph
             isomorphism.  As shown in
             \figref{ComplexityTwo}, only one of these curves (the rightmost) is
             self-overlapping.

             Let $\gamma$ be the unique simple negative subloop of $C$, and let
             $C_1$ be
             obtained from $C$ by a single \Ia-move that contracts $\gamma$.
             We illustrate in~\figref{TwoZeroOne} that the curve $C_1$ is the
             union of two closed positive curves, which can be contracted to $p$
             with a left
             sense-preserving homotopy: first, contract the outer curve to the
             remaining
             intersection point and then contract the inner curve to the root of
             $\gamma$.  

              \begin{figure}[htbp]
                   \centering
                   \includegraphics{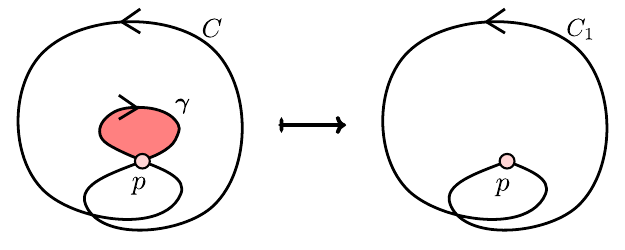}
                   \caption{The curve $C$ satisfies the assumptions of
                       \lemref{SemiSelf} when $C$
                           has exactly two intersection points.  
                           Contracting $\gamma$ with a \Ia-move yields the curve $C_1$.
                                               }\label{fig:TwoZeroOne}
                                           \end{figure}

                                           For an arbitrary self-overlapping
                                           curve, we
                                           consider an immersion
                                           $F:D^2\rightarrow\R^2$, where the
                                           boundary of the disk is
                                           mapped to $[C]$. Let~$\theta \subset
                                           \partial D^2$ have the
                                           image $F(\theta)=[\gamma]$. Let~$p$
                                           be the root of $\gamma$, and let $q$
                                           be any
                                           other point in $\gamma$
                                           whose  preimage is ${q}' \in \theta$.
                                           We obtain a homotopy~$H$ from 
                                           $C$
                                           to ${q}$ by retracting the disk $D^2$
                                           to the point ${q}'$ in 
                                           such a
                                            way that the homotopy fixes $\gamma$
                                            until an
                                            intermediate curve $\widetilde{C}$
                                            is left with only two intersections. 
                                            We know that such a homotopy exists
                                            by the following argument:
                                            at the end of the homotopy, a simple
                                            curve is contracted to a point.  The
                                            last
                                            move before this would either be a
                                            \Ia- or \IIa-move.  
                                            However, since
                                            self-overlapping curves must have at
                                            least two intersection points 
                                            by \obsref{whitneyInvariant}, the
                                            last
                                            move cannot be a \Ia-move since all
                                            intermediate curves induced from a
                                            deformation retraction of $D^2$ are
                                            necessarily self-overlapping.

                                            Notice that the intermediate curve
                                            $\widetilde{C}$ is necessarily the
                                            unique (up
                                            to graph isomorphism)
                                            normal self-overlapping curve
                                            encountered above.  
                                            Let $\widetilde{\gamma}$ be the loop
                                            isomorphic
                                            to $\gamma$ in \figref{TwoZeroOne}.
                                            The curve
                                            $\widetilde{C}_1$ obtained from
                                            $\widetilde{C}$ by contracting
                                            $\widetilde{\gamma}$ with a \Ia-move
                                            is a two-boundary that can contract 
                                            to the root $p$ of
                                            $\widetilde{\gamma}$ via a left
                                            sense-preserving homotopy 
                                            $\widetilde{H}_1$.  We now extend
                                            this to a left sense-preserving
                                            homotopy from
                                            $C_1$ to $p$, where $C_1$ is
                                            obtained from $C$ by contracting
                                            $\gamma$.
                                            Let ${C}_1
                                            \homotopyarrow{\widetilde{H}}
                                            \widetilde{C}$ be the 
                                            sub-homotopy of ${H}$ connecting
                                            ${C}_1$ to~$\widetilde{C}$.
                                            Since $\widetilde{H}$ is induced
                                            from a deformation retraction of
                                            $D^2$, 
                                            the homotopy must be
                                            sense-preserving.
                                            And since we know that
                                            $\widetilde{H_1}$
                                            is left-sense preserving, we know
                                            that $\widetilde{H}$ must be left
                                            sense-preserving (otherwise the
                                            minimal homotopy of
                                            $\widetilde{C}_1$ would be right
                                            sense-preserving as there is only
                                            one positive
                                            self-overlapping curve with two
                                            simple crossing points). 
                                            Finally, we compose these two
                                            homotopies: $C_1 
                                            \homotopyarrow{\widetilde{H}}
                                            \widetilde{C_1}
                                            \homotopyarrow{\widetilde{H_1}} 
                                            p$, which is a left sense-preserving
                                            nullhomotopy with two anchor points.
                                            Hence, by \corref{Kboundary}, 
the 
curve obtained by contracting $\gamma$ via a \Ia-move is a two-boundary.
                                        \end{proof}

Using \lemref{SemiSelf}, we prove Lemmas \ref{lem:RemovingIb}, \ref{lem:RemovingIIb}, and \ref{lem:RemovingIbOrIIb}
below which
provide the key ingredients for proving our main theorem.
\begin{lemma} \label{lem:RemovingIb}
Let $C_0 \homotopyarrow{H_1}  C_1\homotopyarrow{H_2}   p_0$ be a minimum 
homotopy, where $H_1$ consists of a single \Ib-move and $H_2$ is well-behaved. 
If $C_1$ is a positive self-overlapping curve, then $C_0$ is a two-boundary and
 $H_1+H_2$ can be replaced by a well-behaved minimum homotopy.
\end{lemma}

\begin{proof}
We observe that a negative loop is oriented clockwise and a positive loop is oriented counter-clockwise. Hence, a left sense preserving homotopy expands the negative loop and increases the area of the interior face. Similarly, a left sense-reserving homotopy contacts a positive loop and decreases the area of the interior face.

Now, since $H_2$ is well-behaved and $C_1$ is self-overlapping, we know by
\thmref{OneAnchor} that $H_2$ has one anchor point.  Since no contraction
happened in $H_1$, we know that $H_1+H_2$ has only one anchor point and the homotopy is left sense-preserving by \lemref{SensePreserving}.
This implies that $H_1$ creates a negative loop,
since the loop is expanding by 
the homotopy when it is created for the first time. Thus, by \lemref{SemiSelf}, $C_0$ is a two-boundary.
\end{proof}

A similar approach is used to eliminate significant \IIb-moves.

\begin{lemma}[Existence of a Well-Behaved Minimum 
Homotopy]\label{lem:RemovingIIb}
    Let $C_0 \homotopyarrow{H_1} C_1 \homotopyarrow{H_2} p_0$ be a minimum 
homotopy where
$H_1$ consists of a single significant \IIb-move and $H_2$ is well-behaved. 
If
$C_1$ is a two-boundary, then $C_0$ is also a two-boundary. Furthermore, $H_1+H_2$ 
can be
replaced by a well-behaved minimum homotopy $C_0 \homotopyarrow{H} p_0$.
\end{lemma}

\begin{proof}
    {\em(Note: Here, we give a sketch of the proof and leave the technical
    details to the full version of this paper.  We note where details are
    omitted below.)}
    Since~$H_1$ is a single significant \IIb-move, then we know that one of the two
crossing points created is an anchor point. Let's call that point $p_k$.  
Therefore, we have three potential cases,
each of which is illustrated in \figref{anchorSplit}.

\begin{figure}[htbp]
\centering
    \subfloat[Case $1$]{\includegraphics{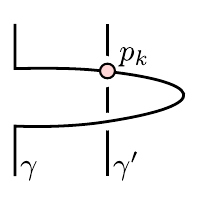}}
    \hspace*{7ex}
    \subfloat[Case $2$]{\includegraphics{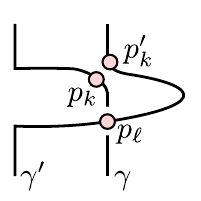}}
    \hspace*{7ex}
    \subfloat[Case $3$]{\includegraphics{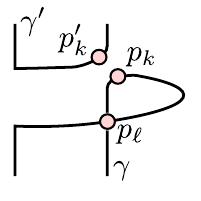}}
\caption{
    We illustrate the three cases that the significant \IIb-move can create in
    \lemref{RemovingIIb}.  To exaggerate the topology of the cases, in Cases $2$
    and $3$, we split the point $p_k$ into two points, $p_k$ and $p_k'$ (even
    though they represent the same point in $\R^2$).
}
\label{fig:anchorSplit}
\end{figure} 

In Case $1$, we notice that splitting at $p_k$ results in only one curve, hence
a contradiction (since if $p_k$ were an anchor point, $p_k$ would be the root of
two curves that form a decomposition of $C_1$).

In Cases $2$ and $3$, we create two subcurves $\gamma$ rooted at $p_k$ and $\gamma'$
rooted at $p_k'$ (as points $p_k=p_k'$; however, we distinguish them for the
purpose of exaggerating the topology).  As illustrated in \figref{anchorSplit},
we can perturb $\gamma$ and $\gamma'$ such that the intersection at $p_k$
disappears.  However, the second intersection created by $H_1$ remains; we will
call this intersection point~$p_{\ell}$.  

In Case $2$, since $\gamma$
and $\gamma'$ are both closed curves, this implies that there must exist at
least one more
intersection point in addition to $p_{\ell}$ (recalling that $p_k$
can be perturbed away).  Let $A$ be the set of intersection points between
$\gamma$ and $\gamma'$ that are also simple crossing points of $C_0$.  
Let $p_i \in A$, and notice 
that there are two curves from~$p_i$ to $p_k$ and two curves from~$p_k$
to $p_i$ such that the union of these four curves is $C$. 
Define curves $\alpha_i$ and $\beta_i$ such 
that~$\alpha_i$ follows $\gamma$ from $p_k$
to $p_i$ then $\gamma'$ from $p_i$ to $p_k'=p_k$ and~$\beta_i$ follows $\gamma'$
from $p_k'=p_k$ to~$p_i$ and then $\gamma$ from~$p_i$ to~$p_k$.
If $C_1$ is a two-boundary, then there exists a $p_i$ such that the 
curves~$\widetilde{\alpha_i}$ and $\widetilde{\beta_i}$ that map to 
$\alpha_i$ and
$\beta_i$ under $H_1$ 
are positive self-overlapping.  The proof of the existence of such an $i$ is
quite technical, and is deferred to the full version of this paper.

In Case $3$, 
we can have two subcases: first, if $\gamma \cap \gamma' = p_k$,
then we let $A$ be the set of intersection crossing points of $\gamma$, and we
can find a $p_i \in A$ using a similar technical argument as required for Case
$2$.
Second, if $|\gamma \cap \gamma'| > 1$, then we follow an argument identical to
the argument for Case $2$.

Now, we define a homotopy $H$ by first contracting $\widetilde{\alpha_i}$ to
$p_i$, and then
contracting $\widetilde{\beta_i}$ to $p_0$.  By \thmref{ReconstructionStable}, we conclude that $H$
is a minimal homotopy and $C_0$ is a two-boundary.
\end{proof}

The following lemma generalizes Lemmas~\ref{lem:RemovingIb} and \ref{lem:RemovingIIb} by removing 
additional assumptions on $H_2$. 

\begin{lemma}\label{lem:RemovingIbOrIIb}
Let $C_0 \homotopyarrow{H_1}  C_1 \homotopyarrow{H_2}  p_0$ be a minimum 
homotopy where
$H_1$ consists of a single \Ib-move or significant \IIb-move, and $H_2$ is well-behaved. 
Then, $H_1+H_2$ can be replaced by a well-behaved minimum homotopy.
\end{lemma}

\begin{proof}
Since $H_2$ is already well-behaved, there is a corresponding decomposition
$\Gamma_1$ of $C_1$ by \thmref{ReconstructionStable}. Let $\gamma$ be the
self-overlapping curve containing the newly created loop by $H_1$. Without loss of generality, assume
that $\gamma$ is positive. 
If $H_1$ consists of a single \Ib-move, then we apply \lemref{RemovingIb} to $\gamma$. The remainder of $H_2$ remains well-behaved.
%

If $H_1$ consists of a significant \IIb-move, then it can be shown by case analysis  that there must be a two-boundary $\gamma'$ that contains $\gamma$, otherwise the \IIb-move is not significant or the homotopy not minimum. We apply \lemref{RemovingIIb} to $\gamma'$ and the remainder of $H_2$ remains well-behaved.
\end{proof}

We are now ready to prove our main theorem.

\begin{theorem}[Main Theorem]\label{thm:main}
Let $C$ be a normal curve. Then:
\begin{itemize}
\item Any minimum homotopy $C\homotopyarrow{H}  p_0$ can be replaced with a 
well-behaved minimum homotopy.
\item There exists a decomposition $\Gamma$ such that the induced homotopy 
$H_{\Gamma}$ is minimum.
\end{itemize}
\end{theorem}
\begin{proof}

We split $H$ into a sequence of subhomotopies $C\homotopyarrow{H_1} 
C_1\homotopyarrow{H_2}  C_2 \ldots \homotopyarrow{H_k}  p_0$, 
$H=\sum_{i=1}^k H_i$
where each subhomotopy consists of a single homotopy move. If none of the moves
is a \Ib-move or a significant \IIb-move, then the homotopy is already
well-behaved. Otherwise, let $C_{j-1} \homotopyarrow{H_j} C_j$ be the 
subhomotopy
containing the last such $b$ move. Then, we can replace
$H_j+H_{j+1}+\ldots+H_k$ with a well-behaved homotopy, $\widetilde{H_j}$, using
\lemref{RemovingIbOrIIb}. The new homotopy
$H_1+H_2+\ldots+\widetilde{H_j}$ is a minimum homotopy which has one less such 
$b$
move. Removing such $b$ moves one by one, we obtain a well-behaved minimum
homotopy.

The second part of the theorem follows from the first part and 
\thmref{ReconstructionStable}.
\end{proof}


It follows that there is a minimum homotopy $H$ such that $E_H$ is constant on each face.
\begin{corollary}   
Let $C$ be a normal curve and let $f_0,f_1,\ldots,f_k$ be the set of faces of
$C$ where $f_0$ is the exterior face. Then, there exists a minimum homotopy $C
\homotopyarrow{H} p_0$ such that $E_H$ is constant on the faces of $C$, i.e., 
if two
points $x,y \in \R^2$ are in the same face, then $E_H(x)=E_H(y)$. Hence, if we
set $E_H(f_i)=E_H(x_i)$ where $x_i \in f_i$, then $E_H(f_0)=0$ and
$$\sigma(C)=\sum_{i=1}^k E_H(f_i) \cdot \textsc{Area}(f_i).$$
\end{corollary} 

	\subsection{Algorithm}
	
Let $C\in\mathfrak{C}_{p_0}$ be a normal curve and let $P_C$ be the set of
intersection points of $C$.  Recall that by \thmref{main} there exists a minimum
homotopy for $C$ that consists of contractions of self-overlapping subcurves
to anchor points. We can therefore check each intersection point to see if it
might serve as an anchor point. If $i\in P_C$ is an intersection point of $C$,
then it breaks $C$ into two subcurves that we denote with $C_{i,1}$ and
$C_{i,2}$. The following recursive formula naively checks all possible ways to
break the curves along their intersection points:
	
	$$ \sigma(C) =  \begin{cases} 
	     W(C) & \text{, if } C \text{ is self-overlapping } \\
	     \min_{i \in P_C } \sigma(C_{i,1} ) + \sigma(C_{i,2}) & \text{, 
otherwise} 
	   \end{cases} 
	$$
	
Using this formula we split $C$ at each intersection point, take the best
split and proceed recursively. In the worst case, this recursive algorithm takes
exponential time in $|P_C|$.

	\section{Conclusions}
We have shown that normal curves admit minimum homotopies that are composed of 
contractions of self-overlapping curves.
%
    At this stage, we have a straight-forward exponential algorithm to compute a minimum homotopy.
But, we are optimistic that  
our structural main theorem lays the foundation for developing  a polynomial-time algorithm. 
In fact, undergraduates Parker Evans and Andrea Burns have developed a tool to visualize
minimum-area homotopies (\url{http://www.cs.tulane.edu/~carola/research/code.html}), which has led us to insights on which we can base an
efficient dynamic programming algorithm.

	Another problem to consider is to find a minimum homotopy between any two normal curves
	with the same end point not just between a curve and its endpoint. For some pair
	of curves $C_1$ and $C_2$, the minimum homotopy area between these curves is
	equal to  the minimum homotopy area of the curve $C_1 \circ C_2^{-1}$.  Also, we
	can extend the minimum homotopy problem to curves on other surfaces. We hope to
	address all these problems in future work. 
	
	\section{Acknowledgments}
	This work has been supported by the National Science Foundation grants
	CCF-1301911, CCF-1618469, and CCF-1618605.  
	We thank Erin Chambers, Jeff Erickson, Parker Evans, and Yusu Wang for their
	useful comments and feedback.
{

\bibliographystyle{plain}

\bibliography{biblio}
}

\newpage
\begin{appendix}
\section{Metric Space} 

\begin{theorem}[Metric Space]
Define 
$\mathfrak{C}_{p_0}^+=\{C \in \mathfrak{C}_{p_0} \mid \whit{C}\geq
0\}/\sim$,
where $C_1\sim C_2$ if $[C_1]=[C_2]$. 
Then $(\mathfrak{C}_{p_0}^+,\sigma)$ is a metric space.
\end{theorem}
\begin{proof}
First, we must show that $\sigma$ is well-defined.
If $C_1,C_2 \in
\mathfrak{C}_{p_0}^+$ and if $C_1 \sim C_2$, where $C_1(\phi(t))=C_2(t)$ for 
some
function $\phi$, then we have a homotopy $C_1\homotopyarrow{H}  C_2$ such 
that
$H(s,t)=C_1(s\phi(t)+(1-s)t)$.
For this homotopy, we have $E_H(x)=1$
if $x \in [C_1]$ and $E_H(x)=0$ otherwise. Since the curve has zero measure, we
have $\textsc{Area}(H)=0$, which gives us $\sigma(C_1,C_2)=0$. Similarly, if $C_1 \sim C_2$ 
and $C'_1 \sim C'_2$ then $\sigma(C_1,C'_1)=\sigma(C_2,C'_2)$.  Hence, 
$\sigma$
is well-defined.

To finish this proof, we must show that $\sigma$ satisfies 
the metric space identities (the identity of indiscernibles, symmetry, and subadditivity).
%
Clearly, $\sigma(C_1,C_2)=0$ if and only if $C_1 \sim C_2$. 
%
If $C_1 
\homotopyarrow{H} C_2$ is a homotopy, then $C_2 
\homotopyarrow{H^{-1}} C_1$ is a homotopy with $H^{-1}(s,t)=H(1-s,t)$ and 
$\textsc{Area}(H)=\textsc{Area}(H^{-1})$. Hence, 
$\sigma(C_1,C_2)=\sigma(C_2,C_1)$. 
%
Finally, if $C_1 \homotopyarrow{H_1} C_2$, 
$C_2\homotopyarrow{H_2} C_3$ and $C_1 \homotopyarrow{H_3} C_3$ are 
minimum 
homotopies, then $\sigma(C_1,C_3)=\textsc{Area}(H_3) \leq 
\textsc{Area}(H_1+H_2) =\textsc{Area}(H_1)+\textsc{Area}(H_2)=\sigma(C_1,C_2)+ 
\sigma(C_2,C_3)$. 
Thus, we conclude that $(\mathfrak{C}_{p_0}^+,\sigma)$ is a metric space.
\end{proof}

\section{Examples} 
In this section, we apply our main theorem to calculate a minimum homotopy for
the curves in \subfigref{NotEqual} and \figref{MinHomEx3}. We say that a set of 
vertices $A=\{i_1,\ldots,i_k\}$ is \emph{valid} if there exists a decomposition 
$\Gamma$ whose set of roots corresponds to the intersection points 
$\{p_{i_1},\ldots,p_{i_k}\}$. 

\begin{exmp}

We check whether the curve in \subfigref{NotEqual} is 
self-overlapping or not. It is not self-overlapping. However, notice that
splitting the curve into two at either intersection point $2$ or $3$  creates
two self-overlapping curves. Hence, there are two different ways to decompose
the curve. $\{0,2\}$ corresponds to the first decomposition $\Gamma_1$   and 
$\{0,3\}$ corresponds to the second
decomposition $\Gamma_2$. Let $H_1$ be the 
homotopy obtained from $\Gamma_1$, and let $H_2$ be the other homotopy, see
\figref{ApplyMain}. We compute $\textsc{Area}(H_1)=2\textsc{Area}(f_2 \cup 
f_1)+\textsc{Area}(f_4)$ and
$\textsc{Area}(H_2)=2\textsc{Area}(f_3 \cup f_1)+\textsc{Area}( f_4)$. The 
smallest of them, in this case
$H_2$, is the minimum homotopy and $\sigma(C)=\textsc{Area}(H_2)$. On the 
other
hand, $W(C)=2\textsc{Area}(f_1)+\textsc{Area}(f_4)$. This shows that 
$\sigma(C)>W(C)$.

\begin{figure}[htbp]
\centering
\includegraphics[height=2 in]{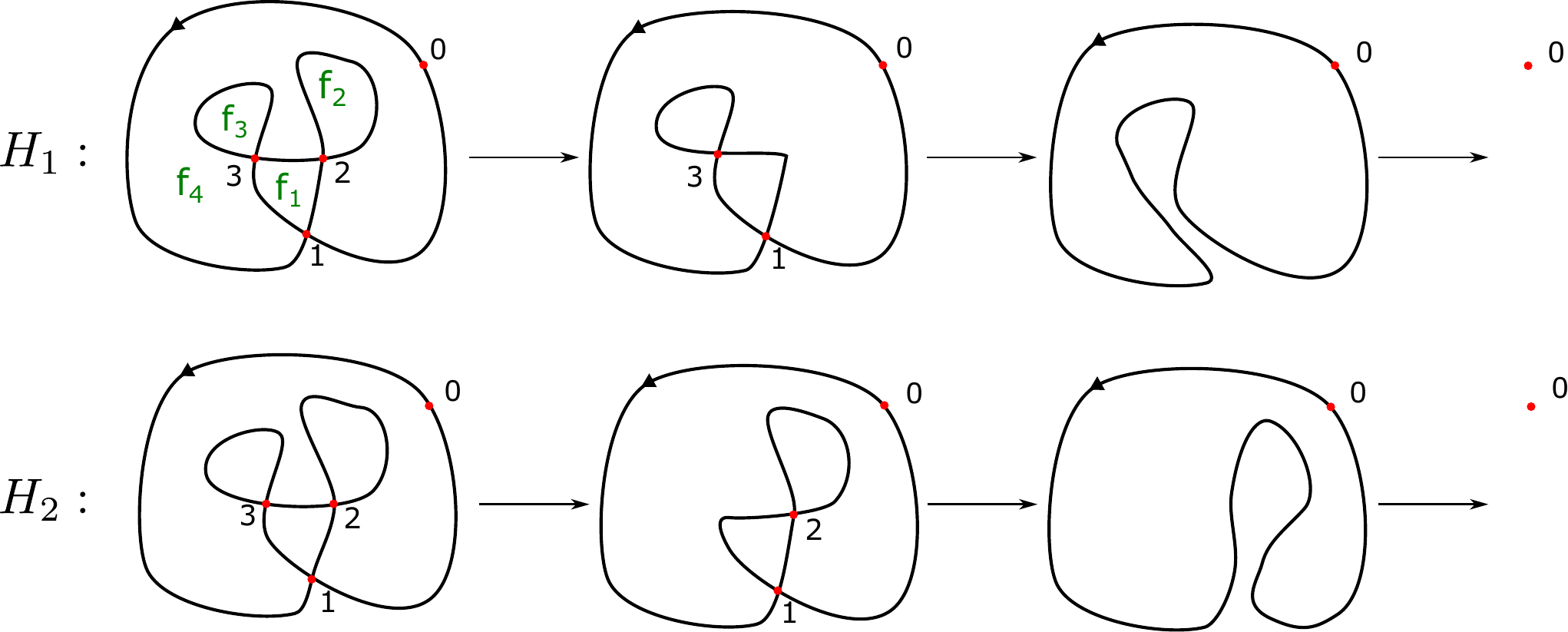}
\caption{Two well-behaved minimum homotopies for the curve in
\subfigref{NotEqual}. Here, $H_2$ sweeps less area since 
$\textsc{Area}(f_3)<\textsc{Area}(f_2)$. }
\label{fig:ApplyMain}
\end{figure} 

\end{exmp}

\begin{exmp}

For the curve in \figref{MinHomEx3}, there are 22 different possible 
decompositions.  We list the valid sets of vertices as follows: $A_1=\{0,3,9\}$, 
$A_2=\{0,3,10\}$, $A_3=\{0,1,2,3,9\}$, $A_4=\{0,1,2,3,10\}$, 
$A_5=\{0,3,4,5,7\}$, $A_6=\{0,3,4,5,8\}$, $A_7=\{0,3,5,7,9\}$ 
$A_8=\{0,3,5,7,10\}$, $A_9=\{0,3,5,8,9\}$, $A_{10}=\{0,3,5,8,10\}$, 
$A_{11}=\{0,1,2,3,4,5,7\}$, $A_{12}=\{0,1,2,3,4,5,8\}$, 
$A_{13}=\{0,1,2,3,5,7,9\}$, $A_{14}=\{0,1,2,3,5,7,10\}$, 
$A_{15}=\{0,1,2,3,5,8,9\}$, $A_{16}=\{0,1,2,3,5,8,10\}$, 
$A_{17}=\{0,3,4,5,6,7,8\}$, $A_{18}=\{0,3,5,6,7,8,9 \}$, 
$A_{19}=\{0,3,5,6,7,8,10\}$, $A_{20}=\{0,1,2,3,4,5,6,7,8\}$, 
$A_{21}=\{0,1,2,3,5,6,7,8,9\}$,  $A_{22}=\{0,1,2,3,5,6,7,8,10\}$.

\begin{figure}[htbp]
\centering
\includegraphics[height=1.4 in]{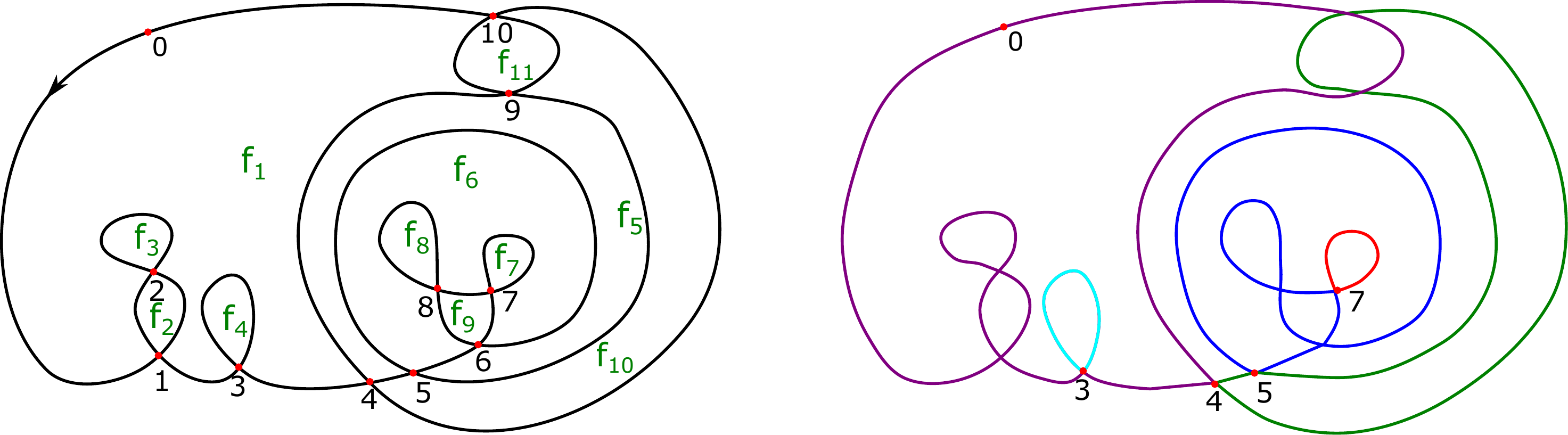}
\caption{The set of anchor indices of the minimum homotopy for this curve is 
$A_H=\{0,3,4,5,7\}$. These vertices decompose the curve into five 
self-overlapping subcurves, the red curve with root $p_7$, the blue curve with 
root $p_5$, the green curve with root $p_4$ and the cyan curve with root $p_3$ 
and the purple curve with root $p_0$.  The homotopy is obtained by contracting 
first the red curve, then the blue curve, then the  green curve,  then the cyan 
curve and finally the purple curve to the corresponding roots $p_7$ , $p_5$, 
$p_4$, $p_3$ and $p_0$ respectively.}
\label{fig:MinHomEx3}
\end{figure}

Among these valid sets, the least area is swept by the homotopy obtained from 
the set $A_5$. Hence, the minimum homotopy area is equal to 
$\sigma(C)=2\textsc{Area}(f_2 \cup f_4 \cup f_7 \cup f_9 \cup f_{11})+ 
\textsc{Area}(f_1 \cup f_6 \cup f_{10})$. Notice that the winding area is equal 
to $W(C)=2\textsc{Area}(f_2 \cup f_4  \cup f_9 )+ \textsc{Area}(f_1 \cup f_6 
\cup f_{10})$ i.e., $\sigma(C)>W(C)$.

\end{exmp}
\section{An Application}

Let $\alpha,\beta$ be two open curves  sharing the same end-points
$\alpha(0)=\beta(0)$ and $\alpha(1)=\beta(1)$. We can concatenate $\alpha$ and
$\beta$ to create a closed curve $C_{\alpha,\beta}$, where
$C_{\alpha,\beta}(t)=\alpha(2t)$ for $t\in[0,\frac{1}{2}]$ and
$C_{\alpha,\beta}(t)=\beta(2-2t)$ for $t\in[\frac{1}{2},1]$. We assume that
$C_{\alpha,\beta}$ is a normal curve, or else we apply a small deformation as
discussed previously.

We define the minimum homotopy area between $\alpha$ and $\beta$ as the minimum
homotopy area of $C_{\alpha,\beta}$, and we denote $\sigma(\alpha,\beta)$. In
other words,  $\sigma(\alpha,\beta)=\sigma(C_{\alpha,\beta})$.

If two curves $\alpha$ and $\beta$ do not share the same endpoints, we create a
closed curve by joining the endpoints via straight lines and define the minimum
homotopy area between  $\alpha$ and $\beta$ as the minimum homotopy area of this
closed curve. See \figref{Concatenate}.

\begin{figure}[htbp]
\centering
\includegraphics[height=1 in]{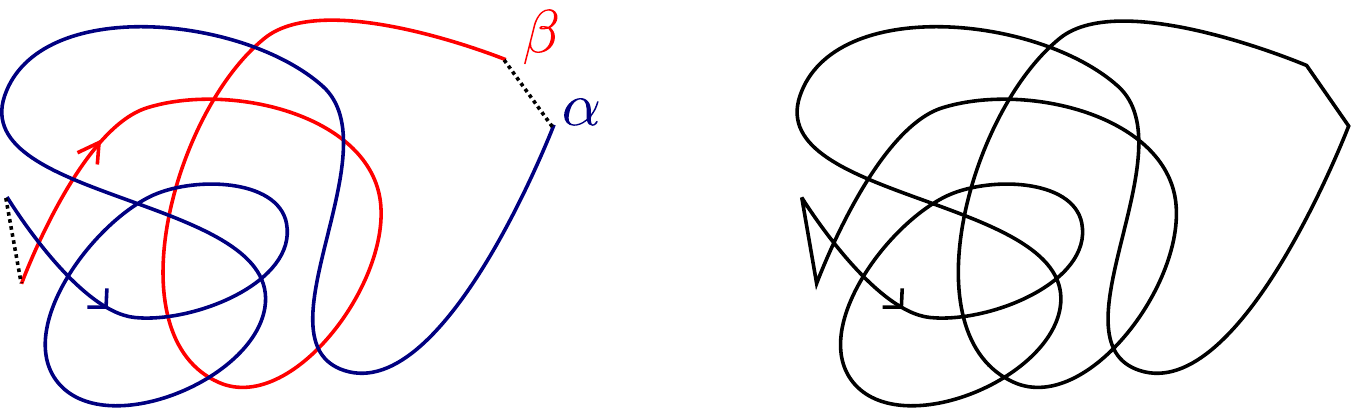}
\caption{On the left, we show two open curves, the blue curve $\alpha$ and 
the red curve $\beta$. On the right, we show the closed curve obtained by 
joining the initial points with straight lines and then concatenating the two 
curves. }
\label{fig:Concatenate}
\end{figure}
Minimum homotopy area can be used to measure the distance between two plane 
graphs, in particular maps created from a set of GPS trajectories.

Let $G_1=(V_1,E_1,w_1)$ and $G_2=(V_2,E_2,w_2)$   be two connected, weighted
plane graphs. We say that a vertex $\widetilde{v} \in V_2$  of $G_2$ is an
associate of $v$ if $$ \| v- \widetilde{v}\|= \underset{w \in V_2}{\min} 
\|v-w\|$$
	
In other words $\widetilde{v}$ is the closest vertex of $G'$ to $v$. We denote
it by $v \sim \widetilde{v}$. For any pair $(u,v) \in V_1 \times V_1 $, $u \neq 
v$,
let $\{p_{\alpha}(u,v)\}_{\alpha \in I}$ be the set of all shortest paths from
$u$ to $v$ and $\{q_{\beta}(u,v)\}_{\beta \in J}$ be the set of all shortest
paths from an associate $\widetilde{u}$ of $u$ to an associate $\widetilde v$ of 
$v$.
Here $p_{\alpha}(u,v)$ is a path $G_1$ and $q_{\beta}(u,v)$ is a path in $G_2$.
Let $C_{\alpha, \beta }(u,v)$ denote the concatenation of $p_\alpha(u,v)$ and
$q_\beta(u,v)$ as defined previously and $$\sigma(u,v)=\underset{\alpha,
\beta}{\min} \text{ } \sigma(C_{\alpha, \beta }(u,v)).$$

And, finally, we define the homotopy area distance between two graphs $G_1$ and $G_2$ as
$$\sigma(G_1,G_2)=\frac{1}{n(n-1)}  {\sum} \sigma(u,vf)$$
where  the summation is taken for each different pair of vertices $u,v  \in V_1$ and $n=|V_1|$.

\end{appendix}

\end{document}